\documentclass[12pt]{article}
\usepackage{mystyle}
\begin{document}
\title{The Ultrapower Axiom implies GCH above a supercompact cardinal}
\author{Gabriel Goldberg}
\maketitle
\begin{abstract}
We prove that the Generalized Continuum Hypothesis holds above a supercompact cardinal assuming the Ultrapower Axiom, an abstract comparison principle motivated by inner model theory at the level of supercompact cardinals.
\end{abstract}

\section{Introduction}
In this paper, we prove that the Generalized Continuum Hypothesis (GCH) holds above a supercompact cardinal assuming the Ultrapower Axiom (UA), an abstract comparison principle motivated by inner model theory at the level of supercompact cardinals:

\begin{thm}[UA]\label{MainThm}
Suppose \(\kappa\) is supercompact. Then for all cardinals \(\lambda\geq \kappa\), \(2^\lambda = \lambda^+\).
\end{thm}

This result is significant in several ways. First, it indicates the possibility that UA along with large cardinals yield some kind of abstract fine structure theory for \(V\) above the least supercompact cardinal. We push this further by proving a combinatorial strengthening of GCH at certain successor cardinals, namely Jensen's \(\diamondsuit\) Principle on the critical cofinality.  

Second, it shows that the eventual GCH follows from a purely ``large cardinal structural" assumption, namely UA + a supercompact, an assumption that has a certain amount of plausibility to it. This is of some philosophical interest since GCH is among the most prominent principles independent of the axioms of set theory.

A third and more technical way in which the theorem is significant is that by proving a very local version of it, we will be able to remove cardinal arithmetic hypotheses from various applications of UA. For example, in \cite{MO}, we prove the linearity of the Mitchell order (essentially) on normal fine ultrafilters on \(P(\lambda)\) assuming UA + \(2^{<\lambda} = \lambda\); in \cite{Frechet}, we remove the hypothesis that \(2^{<\lambda} = \lambda\) using \cref{MainTheorem} and the results of \cite{Frechet}.
\section{Preliminaries}
\begin{defn}
Suppose \(P\) and \(Q\) are inner models. If \(U\) is a \(P\)-ultrafilter, we denote the ultrapower of \(P\) by \(U\) using functions in \(P\) by \(j_U^P : V\to M_U^P\).
\end{defn}
\begin{defn}
If \(i :  P \to Q\) is an elementary embedding, we say \(i\) is:
\begin{enumerate}[(1)]
\item an {\it ultrapower embedding} of \(P\) if there is a \(P\)-ultrafilter \(U\) such that \(i = j_U^P\).
\item an {\it internal ultrapower embedding} of \(P\) if there is a \(P\)-ultrafilter \(U\in P\) such that \(i = j_U^P\).
\end{enumerate}
\end{defn}

\begin{defn}
If \(M_0,M_1,\) and \(N\) are inner models, we write \((i_0,i_1) : (M_0,M_1)\to N\) to denote that \(i_0 : M_0\to N\) and \(i_1 : M_1\to N\) are elementary embeddings.
\end{defn}

\begin{defn}
Suppose \(j_0 : V\to M_0\) and \(j_1 : V\to M_1\) are ultrapower embeddings. A pair of elementary embeddings \((i_0,i_1) : (M_0,M_1)\to N\) is a {\it semicomparison} of \((j_0,j_1)\) if \(i_0\circ j_0 = i_1\circ j_1\) and \(i_1\) is an internal ultrapower embedding of \(M_1\).
\end{defn}

We warn that the notion of a semicomparison is not symmetric: that \((i_0,i_1)\) is a semicomparison of \((j_0,j_1)\) does not imply that \((i_1,i_0)\) is a semicomparison of \((j_1,j_0)\).

\begin{defn}
Suppose \(j_0 : V\to M_0\) and \(j_1 : V\to M_1\) are ultrapower embeddings. A {\it comparison} of \((j_0,j_1)\) is a pair \((i_0,i_1) : (M_0,M_1)\to N\) such that \((i_0,i_1)\) is a semicomparison of \((j_0,j_1)\) and \((i_1,i_0)\) is a semicomparison of \((j_1,j_0)\).
\end{defn}

\begin{ua}
Every pair of ultrapower embeddings admits a comparison.
\end{ua}
\section{Ultrafilters below a cardinal}
The main result of this section is a weakening of the main result of \cite{MO} that is proved without recourse to a cardinal arithmetic hypothesis.

\begin{prp}[UA]\label{MitchellLemma}
Suppose \(\lambda\) is an infinite cardinal and \(M\) is an ultrapower of the universe such that \(M^\lambda\subseteq M\). Then any countably complete ultrafilter on an ordinal less than \(\lambda\) belongs to \(M\).
\end{prp}
Note that the conclusion of \cref{MitchellLemma} is also a consequence of the hypothesis that \(2^{<\lambda} =\lambda\). It is consistent with ZFC, however, that the conclusion of \cref{MitchellLemma} is false. 

In order to prove \cref{MitchellLemma}, we need two preliminary lemmas. The first is the obvious attempt to extend the proof of the linearity of the Mitchell order on normal ultrafilters from \cite{MO} to normal fine ultrafilters.

\begin{lma}\label{FirstAttemptLemma}
Suppose \(j : V\to M\) is an ultrapower embedding and \(\lambda\) is a cardinal such that \(M^\lambda\subseteq M\). Suppose \(D\) is a countably complete ultrafilter on an ordinal \(\gamma \leq \lambda\). Suppose \((k,i) : (M_D,M)\to N\) is a semicomparison of \((j_D,j)\) such that \(k([\textnormal{id}]_D) \in i(j[\lambda])\). Then \(D\in M\).
\begin{proof}
Note that for any \(X\subseteq \gamma\),
\begin{align*}
X\in D&\iff [\text{id}]_D\in j_D(X)\\
&\iff k([\text{id}]_D)\in k(j_D(X))\\
&\iff k([\text{id}]_D)\in i(j(X))\\
&\iff k([\text{id}]_D)\in i(j(X))\cap i(j[\lambda])\\
&\iff k([\text{id}]_D)\in i(j(X)\cap j[\lambda])\\
&\iff k([\text{id}]_D)\in i(j[X])
\end{align*}
Therefore \begin{equation}\label{DInternal}D = \{X\subseteq \gamma : k([\text{id}]_D)\in i(j[X])\}\end{equation} Since \(j\restriction \gamma\in M\), the function defined on \(P(\gamma)\) by \(X\mapsto j[X]\) belongs to \(M\). Moreover \(i\) is an internal ultrapower embedding of \(M\) by the definition of a semicomparison. In particular, \(i\) is a definable subclass of \(M\). Therefore \cref{DInternal} shows that \(D\) is definable over \(M\) from parameters in \(M\), and hence \(D\in M\).  
\end{proof}
\end{lma}

Our next lemma puts us in a position to apply \cref{FirstAttemptLemma}. In this paper, we will only use it in the case \(A = j[\lambda]\), but the general statement is used in the proof of level-by-level equivalence at singular cardinals in \cite{SC}.

\begin{lma}\label{ClubLemma}
Suppose \(\lambda\) is a cardinal, \(j : V\to M\) is an ultrapower embedding, and \( A\subseteq j(\lambda)\) is a nonempty set that is closed under \(j(f)\) for every \(f : \lambda\to \lambda\). Suppose \(D\) is a countably complete ultrafilter on an ordinal \(\gamma < \lambda\). Suppose \((i,k) : (M,M_D)\to N\) is a semicomparison of \((j,j_D)\). Then \(k([\textnormal{id}]_D) \in i(A)\).
\begin{proof}
Let \(A' = k^{-1}[i(A)]\). By the definition of a semicomparison, \(k : M_D\to N\) is an internal ultrapower embedding, and therefore \(A'\in M_D\). We must show that \([\text{id}]_D\in A'\). 

We first show that \(j_D[\lambda]\subseteq A'\). Note that \(j[\lambda]\subseteq A\) since \(A\) is nonempty and closed under \(j(c_\alpha)\) for any \(\alpha < \lambda\), where \(c_\alpha : \lambda\to \lambda\) is the constant function with value \(\alpha\). Thus \(i\circ j[\lambda]\subseteq i(A)\). Since \((i,k)\) is a semicomparison, \(k\circ j_D[\lambda] = i\circ j[\lambda]\subseteq i(A)\). So \(j_D[\lambda]\subseteq k^{-1}[i(A)] = A'\). 

We now show that \(A'\) is closed under \(j_D(f)\) for any \(f : \lambda\to \lambda\). Fix \(\xi \in A'\) and \(f : \lambda\to \lambda\); we will show \(j_D(f)(\xi)\in A'\). By assumption \(A\) is closed under \(j(f)\), and so by elementarity \(i(A)\) is closed under \(i(j(f))\). In particular, since \(k(\xi)\in i(A)\), \(i(j(f))(k(\xi))\in i(A)\). But \(i(j(f))(k(\xi)) = k(j_D(f)(\xi))\). Now \(k(j_D(f)(\xi))\in i(A)\) so \(j_D(f)( \xi)\in k^{-1}(i(A)) = A'\), as desired.

Since \(\gamma < \lambda\) and \(j_D[\lambda]\subseteq A'\), \(A'\) contains \(j_D[\gamma^+]\). Thus \(A'\) is cofinal in the \(M_D\)-regular cardinal \(j_D(\gamma^+) = \sup j_D[\gamma^+]\). In particular, \(|A'|^{M_D}\geq j_D(\gamma^+)\). Fix \(\langle B_\xi : \xi < \gamma\rangle\) with \(A' = j_D(\langle B_\xi : \xi < \gamma\rangle)_{[\text{id}]_D}\). We may assume without loss of generality that \(B_\xi\subseteq \lambda\) and \(|B_\xi| \geq \gamma^+\) for all \(\xi < \gamma\). Therefore there is an injective function \(g :\gamma \to \lambda\) such that \(g(\xi)\in B_\xi\) for all \(\xi < \gamma\). By Los's theorem, \(j_D(g)([\text{id}]_D)\in A'\). Let \(f : \lambda\to \lambda\) be a function satisfying \(f(g(\xi)) = \xi\) for all \(\xi < \gamma\). Now \(A'\) is closed under \(j_D(f)\). But \(j_D(f)(j_D(g)([\text{id}]_D)) = [\text{id}]_D\). Therefore \([\text{id}]_D\in A'\), as desired.
\end{proof}
\end{lma}

\cref{MitchellLemma} now follows easily.
\begin{proof}[Proof of \cref{MitchellLemma}]
Fix a countably complete ultrafilter \(D\) on an ordinal less than \(\lambda\). By UA, there is a comparison \((k,i) : (M_D,M)\to N\) of \((j_D,j)\). By \cref{ClubLemma} with \(A = j[\lambda]\), \(k([\text{id}]_D)\in i(j[\lambda])\). Therefore the hypotheses of \cref{FirstAttemptLemma} are satisfied, so \(D\in M\), as desired.
\end{proof}

\section{A proof of GCH}
\cref{MainThm} above follows immediately from the following more local statement:

\begin{thm}[UA]\label{MainTheorem}
Suppose \(\kappa \leq \delta\) are cardinals with \(\kappa \leq \textnormal{cf}(\delta)\). If \(\kappa\) is \(\delta^{++}\)-supercompact, then for any cardinal \(\lambda\) with \(\kappa\leq \lambda \leq\delta^{++}\), \(2^\lambda = \lambda^+\).
\end{thm}

Combining \cref{MainTheorem} with the results of \cite{SC}, the hypothesis that \(\kappa\) is \(\delta^{++}\)-supercompact can be weakened to the assumption that \(\kappa\) is \(\delta^{++}\)-strongly compact.

The trickiest part of the proof is the following fact.
\begin{thm}[UA]\label{HardPart}
Suppose \(\kappa\) and \(\delta\) are cardinals such that \(\textnormal{cf}(\delta)\geq \kappa\). Suppose \(\kappa\) is \(\delta^{++}\)-supercompact. Then \(2^\delta = \delta^+\).
\begin{proof}
Assume towards a contradiction that \(2^\delta > \delta^+\). 
\begin{clm} For every \(A\subseteq \delta^{++}\), there is a normal fine \(\kappa\)-complete ultrafilter \(U\) on \(P_\kappa(\delta)\) with \(A\in M_U\). \end{clm}
\begin{proof} The argument for this is due to Solovay (\cite{Kanamori}, Theorem 3.8). Assume towards a contradiction that the claim fails. Let \(j : V\to M\) be an elementary embedding such that \(M^{\delta^{++}}\subseteq M\) and \(j(\kappa) > \delta^{++}\). The claim then fails in \(M\) since \(P(\delta^{++})\subseteq M\). Let \(W\) be the normal fine countably complete ultrafilter on \(P_\kappa(\delta)\) derived from \(j\) using \(j[\delta]\).

Let \(k : M_W\to M\) be the factor embedding, so \(k\circ j_W = j\) and \(k(j_W[\delta]) = j[\delta]\). Then \[\textsc{crt}(k) > (2^\delta)^{M_W}\geq \delta^{++}\] Since \(k[M_W]\) is an elementary substructure of \(M\) and the parameters \(\kappa\) and \(\delta^{++}\) belong to \(k[M_W]\), the claim fails in \(M\) for some \(A\in k[M_W]\). Since \(\textsc{crt}(k) > \delta^{++}\), \(k^{-1}(A) = A\), so the claim fails in \(M\) for some \(A\in M_W\). But \(W\in M\) by \cref{MitchellLemma}, and \(W\) witnesses that the claim is true for \(A\) in \(M\). This is a contradiction.
\end{proof}

Let \(W\) be a \(\delta^+\)-supercompact ultrafilter on \(\delta^+\) with \(j_W(\kappa) > \delta^+\). We claim \(P(\delta^{++})\subseteq M_W\). Suppose \(A\subseteq \delta^{++}\). For some normal fine \(\kappa\)-complete ultrafilter \(U\) on \(P_\kappa(\delta)\), \(A\in M_U\). But since \(|P_\kappa(\delta)|= \delta\), \(U\in M_W\) by \cref{MitchellLemma}. It is easy to see that this implies \(A\in M_W\).

Let \(Z\) be a \(\delta^{++}\)-supercompact ultrafilter on \(\delta^{++}\) with \(j_Z(\kappa) > \delta^{++}\). Let \(k : M_W\to N\) be the ultrapower of \(M_W\) by \(Z\) using functions in \(M_W\). Fix an ultrafilter \(D\) on \(\delta^+\) such that \(M_D = N\) and \(j_D = k\circ j_W\). 

Since \(P(\delta^{++})\subseteq M_W\), \((V_\kappa)^{\delta^{++}}\subseteq M_W\). Therefore letting \(\kappa' = j_Z(\kappa) = k(\kappa)\), we have \(N\cap V_{\kappa'} = M_Z\cap V_{\kappa'}\). By \cref{MitchellLemma}, \(D\in M_Z\), and therefore \(D\in M_Z\cap V_{\kappa'}\). It follows that \(D\in N = M_D\), a contradiction.
\end{proof}
\end{thm}

The following lemma is essentially due to Solovay.

\begin{lma}[UA]\label{SolovayEasy}
Suppose \(\kappa\leq \delta\) are cardinals, \(\textnormal{cf}(\delta)\geq \kappa\) and \(\kappa\) is \(2^\delta\)-supercompact. Then \(2^{2^\delta} = (2^\delta)^+\).
\begin{proof}
By the same argument of Solovay (\cite{Kanamori}, Theorem 3.8), there are \(2^{2^\delta}\) normal fine ultrafilters on \(P_\kappa(\delta)\). Applying \cref{HardPart}, we have \(2^{<\delta} = \delta\), and so we can apply the main theorem of \cite{MO}, to conclude that these \(2^{2^\delta}\) normal fine ultrafilters on \(P_\kappa(\delta)\) are linearly ordered by the Mitchell order. The Mitchell order has rank at most \((2^\delta)^+\), so \(2^{2^\delta}\leq (2^\delta)^+\).
\end{proof}
\end{lma}

Regarding this lemma, a much more complicated argument in \cite{Frechet} shows that under UA, a set \(X\) carries at most \((2^{|X|})^+\) countably complete ultrafilters. With \cref{SolovayEasy} in hand, we can prove the main theorem of this paper.
\begin{proof}[Proof of \cref{MainTheorem}]
Suppose \(\lambda\) is a cardinal with \(\kappa \leq \lambda \leq \delta^{++}\). 
\begin{case} \(\lambda\leq \delta\)\end{case}
If \(\lambda\) is regular then by \cref{HardPart}, \(2^\lambda = \lambda^+\). If \(\lambda\) is singular then \(2^{<\lambda} = \lambda\) by \cref{HardPart}, so \(2^\lambda = \lambda^+\) by the local version of Solovay's theorem \cite{Solovay}.
\begin{case}\label{Last} \(\lambda =\delta^+\). \end{case}
Since \(\kappa\) is \(\delta^+\)-supercompact and \(2^\delta = \delta^+\), \(\kappa\) is \(2^\delta\)-supercompact. Therefore by \cref{SolovayEasy}, \(2^{2^\delta} = (2^\delta)^+\). In other words, \(2^{(\delta^+)} = \delta^{++}\). 
\begin{case} \(\lambda =\delta^{++}\) \end{case}
Given that \(2^{(\delta^+)} =\delta^{++}\) by \cref{Last}, the case that \(\lambda = \delta^{++}\) can be handled in the same way as \cref{Last}.
\end{proof}

\begin{cor}[UA]
Suppose \(\kappa\leq \delta\) and \(\kappa\) is \(2^\delta\)-supercompact. Then \(2^\delta = \delta^+\).
\begin{proof}
If \(\delta\) is singular this follows from Solovay's theorem \cite{Solovay}. Assume instead that \(\delta\) is regular. Assume towards a contradiction that \(2^\delta \geq \delta^{++}\). Then \(\kappa\) is \(\delta^{++}\)-supercompact, so by \cref{MainTheorem}, \(2^\delta = \delta^+\), a contradiction.
\end{proof}
\end{cor}

Let us point out another consequence that one can obtain using a result in \cite{SC}:

\begin{thm}[UA]
Suppose \(\nu\) is a cardinal and \(\nu^+\) carries a countably complete uniform ultrafilter. Then \(2^{<\nu} = \nu\).
\begin{proof}
By one of the main theorems of \cite{SC}, some cardinal \(\kappa \leq\nu\) is \(\nu^+\)-supercompact. If \(\kappa = \nu\) then obviously \(2^{<\nu} = \nu\). So assume \(\kappa < \nu\). If \(\nu\) is a limit cardinal, then the hypotheses of \cref{MainTheorem} hold for all sufficiently large \(\lambda < \nu\) and hence GCH holds on a tail below \(\nu\), so \(2^{<\nu} = \nu\). So assume \(\nu = \lambda^+\) is a successor cardinal. If \(\lambda\) is singular, then \(\lambda\) is a strong limit singular cardinal by \cref{MainTheorem}, so \(2^\lambda = \lambda^+\) by Solovay's theorem \cite{Solovay}, and hence \(2^{<\nu} = \nu\). Finally if \(\lambda\) is regular, we can apply \cref{MainTheorem} directly to conclude that \(2^\lambda = \lambda^+\), so again \(2^{<\nu} = \nu\). 
\end{proof}
\end{thm}

This leaves open some questions about further localizations of the GCH proof.
\begin{qst}[UA]
Suppose \(\kappa\) is \(\delta\)-supercompact. Must \(2^\delta = \delta^+\)?
\end{qst}
We conjecture that it is consistent with UA that \(\kappa\) is measurable but \(2^\kappa > \kappa^+\), which would give a negative answer in the case \(\kappa = \delta\). In certain cases, the question has a positive answer as an immediate consequence of our main theorem:
\begin{prp}[UA]
Suppose \(\kappa \leq \lambda\), \(\textnormal{cf}(\lambda) = \omega\), and \(\kappa\) is \(\lambda\)-supercompact. Then \(2^\lambda = \lambda^+\).

Suppose \(\kappa \leq \lambda\), \(\omega_1 \leq \textnormal{cf}(\lambda) < \lambda\), and \(\kappa\) is \({<}\lambda\)-supercompact. Then \(2^\lambda = \lambda^+\).

Suppose \(\kappa \leq \lambda\), \(\lambda\) is the double successor of a cardinal of cofinality at least \(\kappa\), and \(\kappa\) is \(\lambda\)-supercompact. Then \(2^\lambda = \lambda^+\).\qed
\end{prp}

Another interesting localization question is the following:
\begin{qst}[UA]
Suppose \(\kappa\) is the least ordinal \(\alpha\) such that there is an ultrapower embedding \(j : V\to M\) with \(j(\alpha) > (2^{\kappa})^+\). Must \(2^\kappa = \kappa^+\)?
\end{qst}

\section{\(\diamondsuit\) on the critical cofinality}
We conclude with the observation that stronger combinatorial principles than GCH follow from UA.

\begin{thm}[UA]\label{Diamond}
Suppose \(\kappa\) is \(\delta^{++}\)-supercompact where \(\textnormal{cf}(\delta)\geq \kappa\). Then \(\diamondsuit(S^{\delta^{++}}_{\delta^+})\) holds.
\end{thm}
For the proof, we need a theorem of Kunen.
\begin{defn}
Suppose \(\lambda\) is a regular uncountable cardinal and \(S\subseteq \lambda\) is a stationary set. Suppose \(\langle \mathcal A_\alpha : \alpha \in S\rangle\) is a sequence of sets with \(\mathcal A_\alpha \subseteq P(\alpha)\) and \(|A_\alpha| \leq \alpha\) for all \(\alpha < \lambda\). Then \(\langle \mathcal A_\alpha : \alpha \in S\rangle\) is a \(\diamondsuit^-(S)\)-sequence if for all \(X\subseteq \lambda\), \(\{\alpha \in S : X\cap \alpha\in \mathcal A_\alpha\}\) is stationary.
\end{defn}
\begin{defn}
\(\diamondsuit^-(S)\) is the assertion that there is a \(\diamondsuit^-(S)\)-sequence.
\end{defn}

\begin{thm}[Kunen, \cite{Kunen}]\label{KunenDiamond}
Suppose \(\lambda\) is a regular uncountable cardinal and \(S\subseteq \lambda\) is a stationary set. Then \(\diamondsuit^-(S)\) is equivalent to \(\diamondsuit(S)\).\qed
\end{thm}

\begin{proof}[Proof of \cref{Diamond}]
By \cref{MainTheorem}, GCH holds on the interval \([\kappa,\delta^{++}]\), and we will use this without further comment.

For each \(\alpha < \delta^{++}\), let \(\mathcal U_\alpha\) be the unique ultrafilter of rank \(\alpha\) in the Mitchell order on normal fine \(\kappa\)-complete ultrafilters on \(P_\kappa(\delta)\). The uniqueness of \(\mathcal U_\alpha\) follows from the main theorem of \cite{MO}. Let \(\mathcal A_\alpha = P(\alpha)\cap M_{\mathcal U_\alpha}\). Note that \(|\mathcal A_\alpha| \leq \kappa^\delta = \delta^+\). Let \[\vec{\mathcal A} = \langle \mathcal A_\alpha : \alpha < \delta^{++}\rangle\] Note that \(\vec{\mathcal A}\) is definable in \(H_{\delta^{++}}\) without parameters.

\begin{clm} \(\vec{\mathcal A}\) is a \(\diamondsuit^-(S^{\delta^{++}}_{\delta^+})\)-sequence.\end{clm}
\begin{proof}
Suppose towards a contradiction that \(\vec{\mathcal A}\) is not a \(\diamondsuit^-(S^{\delta^{++}}_{\delta^+})\)-sequence. Let \(\mathcal W\) be a \(\kappa\)-complete normal fine ultrafilter on \(P_\kappa(\delta^{++})\). Then in \(M_\mathcal W\),  \(\vec{\mathcal A}\) is not a \(\diamondsuit^-(S^{\delta^{++}}_{\delta^+})\)-sequence. Let \(\mathcal U\) be the \(\kappa\)-complete normal fine ultrafilter on \(\delta\) derived from \(\mathcal W\) and let \(k : M_\mathcal U \to M_\mathcal W\) be the factor embedding. Let \(\gamma = \textsc{crt}(k) = \delta^{++M_\mathcal U}\).

Since  \(\vec{\mathcal A}\) is definable in \(H_{\delta^{++}}\) without parameters, \(\vec {\mathcal A}\in \text{ran}(k)\). Therefore \(k^{-1}(\vec {\mathcal A}) = \vec {\mathcal A}\restriction \gamma\) is not a \(\diamondsuit^-(S^{\gamma}_{\delta^+})\)-sequence in \(M_\mathcal U\). Fix a witness \(A\in P(\gamma)\cap M_\mathcal U\) and a closed unbounded set \(C\in P(\gamma)\cap M_\mathcal U\) such that for all \(\alpha \in C\cap S^{\gamma}_{\delta^+}\), \(A\cap \alpha\notin \mathcal A_\alpha\). By elementarity, for all \(\alpha\in k(C)\cap S^{\delta^{++}}_{\delta^+}\), \(k(A)\cap \alpha\notin \mathcal A_\alpha\). Since \(\mathcal U\) is \(\delta\)-supercompact, \(\text{cf}(\gamma) = \delta^+\), and so in particular \(k(A)\cap \gamma\notin \mathcal A_\gamma\). Since \(\gamma = \textsc{crt}(k)\), this means \(A\notin \mathcal A_\gamma\).

Note however that \(\mathcal U\) has Mitchell rank \(\delta^{++M_\mathcal U} = \gamma\), so \(\mathcal U = \mathcal U_\gamma\). Therefore \(\mathcal A_\gamma = P(\gamma)\cap M_\mathcal U\), so \(A\in \mathcal A_\gamma\) by choice of \(A\). This is a contradiction.
\end{proof}
By \cref{KunenDiamond}, this completes the proof.
\end{proof}

\section{On the linearity of the Mitchell order}
We close with the question of whether GCH follows from the linearity of the Mitchell order alone. To pose the question, we must first formulate the strongest statement of the linearity of the Mitchell order that we can prove from UA.
\begin{defn}
A uniform ultrafilter \(U\) on a cardinal is {\it seed-minimal} if \([\text{id}]_U\) is the least \(\alpha\in \text{Ord}^{M_U}\) such that \(M_U = H^{M_U}(j_U[V]\cup \{\alpha\})\). 
\end{defn}

In other words, \(U\) is minimal if no regressive function is one-to-one on a \(U\)-large set. Using the Axiom of Choice, it is easy to prove that any countably complete ultrafilter is isomorphic to a unique seed-minimal ultrafilter.

\begin{defn}
A {\it generalized normal ultrafilter} is a seed-minimal ultrafilter that is isomorphic to a normal fine ultrafilter on \(P(X)\) for some set \(X\). 
\end{defn}

\(U\) is a generalized normal ultrafilter on \(\lambda\) if and only if \(U\) is weakly normal and \(M_U\) is closed under \(\lambda\)-sequences. One of the main theorems of \cite{Frechet} states that UA implies that the Mitchell order is linear on generalized normal ultrafilters. This hypothesis is strictly weaker than UA, but one would expect it to be quite powerful in the context of a supercompact cardinal.

\begin{qst}
Assume the Mitchell order is linear on generalized normal ultrafilters. Does \(2^\lambda = \lambda^+\) for all \(\lambda\) greater than or equal to the least supercompact cardinal?
\end{qst}

\bibliography{Bibliography}{}
\bibliographystyle{unsrt}

\end{document}